\documentclass[11pt]{amsart}

\usepackage[usenames,dvipsnames,svgnames,table]{xcolor} 
\usepackage{amssymb,amsfonts,amsrefs}
\usepackage{amsmath,amsthm}

\usepackage[active]{srcltx}		
\usepackage{pdfsync}					
\usepackage[colorinlistoftodos]{todonotes}
\newcommand{\dd}{\,{\rm d}}

\newcommand\R{{\mathbb{R}}}
\renewcommand\P{{\mathbb{P}}}
\newcommand\N{{\mathbb{N}}}

\renewcommand\div{{\rm div}}

\newtheorem{theorem}{Theorem}[section]
\newtheorem{proposition}[theorem]{Proposition}
\newtheorem{lemma}[theorem]{Lemma}

\theoremstyle{definition}

\theoremstyle{remark}
\newtheorem{remark}[theorem]{Remark}
\numberwithin{equation}{section}

\begin{document}

\title[Large time energy growth for the Boussinesq system]
{A short proof of the large time energy growth for the Boussinesq system}

\author{Lorenzo Brandolese}
\address{L. Brandolese: Univ Lyon, Universit\'e Claude Bernard Lyon 1, CNRS UMR 5208, Institut Camille Jordan, 43 blvd. du 11 novembre 1918, F-69622 Villeurbanne cedex, France.}
\email{brandolese{@}math.univ-lyon1.fr}
\urladdr{http://math.univ-lyon1.fr/$\sim$brandolese}

\author{Charafeddine Mouzouni}
\address{C. Mouzouni: Univ Lyon, \'Ecole centrale de Lyon, CNRS UMR 5208, Institut Camille Jordan, 36 avenue Guy de Collonge, F-69134 Ecully Cedex, France.}
\email{mouzouni{@}math.univ-lyon1.fr}

\subjclass[2000]{Primary 76D05; Secondary 35B40}

\date{\today\\ \quad 
This work was performed within the framework of the LABEX MILYON (ANR-10-LABX-0070) of Université de Lyon, within the program ``Investissements d'Avenir'' (ANR-11-IDEX-0007) operated by the French National Research Agency (ANR).
L. Brandolese was also supported by project ANR-13-BS01-0003 - DYFICOLTI - DYnamique des Fluides, Couches Limites, Tourbillons et Interfaces.}

\begin{abstract}

We give a direct proof of the fact that the $L^p$-norms of global solutions of the Boussinesq system
in~$\R^3$ grow large as $t\to\infty$ for $1<p<3$ and decay to zero for $3<p\le\infty$, providing exact estimates from below and above using a suitable decomposition of the space-time space $\R^{+}\times\R^{3}$. In particular,  the kinetic energy blows up as $\|u(t)\|_2^2\sim ct^{1/2}$ for large time.
This constrasts with the case of the Navier--Stokes equations.
\end{abstract}

\keywords{Kato spaces, the Boussinesq equation and asymptotic behavior}

\maketitle


\section{Introduction}

The incompressible Boussinesq system describes the dynamics of an incompressible fluid, taking into account heat exchanges. In addition to the flow, the transport and diffusion of temperature, we have also convection currents created by the vertical force of buoyancy. The Boussinesq approximation consists in neglecting the variations of the density in the continuity equation. Accordingly with this approximation, we also neglect
the local heat source due to the viscous dissipation. Taking all the physical constants equal to $1$ we can write the Boussinesq system in the following form,
\begin{equation}
\label{B} 
\left\{
\begin{aligned}
 &\partial_t \theta +u \cdot \nabla \theta =  \Delta \theta\\
 &\partial_t u +u\cdot\nabla u+\nabla p=\Delta u+\theta e_3\\
  &\nabla\cdot u=0\\
 &u|_{t=0}=u_0,\;\,\theta|_{t=0}=\theta_0
\end{aligned}
\right.
\qquad x\in \R^3, t\in \R_{+}.
\end{equation}
Here $u\colon\R^3\times\R^+\to \R^3$ is the velocity field and the scalar field
$\theta\colon\R^3\times\R^+\to\R$ denotes the temperature. The function $p\colon \R^3\times\R^+\to \R$
is the Lagrange multiplier related to the constraint of incompressibility, 
and is closely related to the pressure of the flow.
Moreover, $e_3=(0,0,1)$ is a constant unit vertical vector.

In the particular case of incompressible Navier Stokes system ($\theta\equiv0$), starting with T. Kato~\cite{Kato} and M.~E.~Schonbek~\cite{Schon84}, many authors studied the decay problem of various $L^p$ norms as $t\to+\infty$.
For example, it is known that if $\|u_0\|_3$ is small enough and if $\int |u_0(x)|(1+|x|)\mathrm{d}x<\infty$, then one has
$$
\Vert u(t)\Vert_{p} \leq Ct^{-2+3/(2p)},\qquad \mbox{for } t>0 \mbox{ and } 1\leq p \leq \infty 
$$
for the solutions of the Navier--Stokes equations.
Asymptotic profiles like those constructed in Y.~Fujigaki and T.~Miyakawa~\cite{Japan} prove the optimality of such decay rates.
In particular this shows that the kinetic energy vanishes when $t\rightarrow \infty$.

The asymptotic behavior of the $L^{p}$-norms for Boussinesq system was expected to be similar to the particular case of the incompressible Navier Stokes system\footnote{Indeed, it was claimed in~\cite{ResulFaux} that $\Vert u(t)\Vert_{2}\rightarrow \infty$ for solutions of Boussinesq system.
However, this article contained an erratum that was pointed out in \cite{BraS12}.}.
But the main result in \cite{BraS12} put in evidence that this is true only if the initial temperature has zero mean: indeed, if the data satisfies appropriate  size and localization conditions then one has, for $t$ large enough,
\begin{equation}\label{ingrech}
c(1+t)^{-1/2+3/2p} \leq \Vert u(t)\Vert_{p} \leq C(1+t)^{-1/2+3/2p}, \qquad \mbox{ for } 1<p\leq \infty,
\end{equation}
with $c>0$ if and only if $\int \theta_{0}\neq 0 $. 
This proves in particular that $\Vert u(t)\Vert_{L^{2}}\sim ct^{1/4}$ if the condition $\int \theta_{0}\neq 0$ is satisfied. 
The drawbacks of the approach of~\cite{BraS12} are that the proof of~\eqref{ingrech} is involved and
the conditions on the data too restrictive. For example, the smallness assumptions needed in
\cite{BraS12} was of the form $\mathrm{ess\,sup}_{\R^3}{|x|\,|u_0(x)|}+\mathrm{ess\,sup}_{\R^3}|x|^3|\theta_0(x)|+\|\theta_0\|_1<\epsilon$.
Moreover, the data were assumed to satisfy additional pointwise estimates.
It does not look natural to put such restrictive \emph{pointwise} conditions for studying the long time behaviour of $L^p$-norms. A similar remark could apply to the conditions proposed in~\cite{Wen16}, where the results of~\cite{BraS12} are extended to higher-order derivatives, but assuming the data in the Schwartz class.

In this paper we aim to give a simpler and shorter proof for \eqref{ingrech} by assuming much 
weaker assumptions on the data and using the natural functional settings, inspired by Kato's classical work \cite{Kato} for the incompressible Navier Stokes system.
The subtle part will be the proof of the lower bounds: the classical approach (see \cite{Japan} and the references therein) of writing
an asymptotic profile for $u$ putting in evidence the leading term as $t\to\infty$ is not applicable
to our situation. Indeed, it turns out that, when $\int\theta_0\not=0$, the linear and the nonlinear terms behave at the same rates in $L^p$.
The main idea will be to split the analysis into different regions of space-time: the size of the linear terms contributing to~$u$ will be more important than the size of the nonlinear terms
in some of these regions, namely, in $\{(x,t)\colon |x|\ge A\sqrt t\}$ provided $A$ is large enough.
A careful linear analysis will finally lead to the lower bound in~\eqref{ingrech}. 

The use of Duoandikoetxea and Zuazua decompositions for distributions, and of variants of Young-type estimates (useful for estimating convolution integrals outside balls of large radii)
are two original features of this paper. 
The crucial technical step is carried in Proposition~\ref{pr:kato-r} and the main result is stated in Theorem~\ref{th:lube}.

\section{The Boussinesq equations}

The integral formulation of the Boussinesq system \eqref{B} reads:
\begin{equation}
 \label{IE}
\left\{
\begin{aligned}
 &\theta(t)=e^{t\Delta}\theta_0-\int_0^t e^{(t-s)\Delta}\nabla \cdot(\theta u)(s)\dd s\\
 &u(t)=e^{t\Delta}u_0-\int_0^t e^{(t-s)\Delta}\P\nabla \cdot(u\otimes u)(s)\dd s
+\int_0^t e^{(t-s)\Delta}\P \theta(s) e_3\dd s.\\
&\nabla \cdot u_0=0
\end{aligned}
\right.
\end{equation}
Here $\P$ denotes the projector on the space of divergence-free fields, which is also called Leray's projector.
We will not discuss here the issue of the equivalence between the system \eqref{IE} and the original
system~\eqref{B}. One could see \cite[Theorems 1.1-1.2]{Lem02} for a discussion of this issue in the particular case of incompressible Navier Stokes system. In all this paper the Boussinesq system will be treated in the integral form above.

Let us write the unknown as ${\bf v}=\displaystyle\binom{u}{\theta}$.  It is convenient to rewrite the above integral system in the following abstract form:
\begin{equation}
\label{absy}
{\bf v}={\bf a}+\bf {B}({\bf v},{\bf v}),
\end{equation}
where ${\bf B}\colon E\times E\to E$ is a bilinear operator in a suitable Banach space~$E$, 
and ${\bf a}\in E$ is given  in terms of the initial data.
Here $E=X\times Y$ where $X$ denotes the Banach space of the velocity and $Y$ the Banach space of the temperature.
We define $X$ to be the space of all $C([0,\infty),L^3)$ divergence-free vector fields~$u$ such that $\|u\|_X<\infty$, and $Y$ the space of all $C([0,\infty),L^1)$ functions such that $\|\theta\|_Y<\infty$.
Here,
\begin{align}
 & \|u\|_X\equiv \sup_{t>0}\|u(t)\|_3+\sup_{t>0}\sqrt t\|u(t)\|_\infty,\\
 & \|\theta\|_Y\equiv\sup_{t>0}\|\theta(t)\|_1 + \sup_{t>0} t^{3/2} \|\theta(t)\|_\infty.
\end{align}

To write~\eqref{IE} in the form~\eqref{absy}, replace the equation of~$\theta$ inside 
$\int_0^t e^{(t-s)\Delta}\P \theta(s) e_3\dd s$.
A simple computation transforms~\eqref{IE} into the equivalent system
\begin{equation}
 \label{IEE}
\left\{
\begin{aligned}
 &\theta(t)=e^{t\Delta}\theta_0-\int_0^t e^{(t-s)\Delta}\nabla \cdot(\theta u)(s)\dd s\\
 &u(t)=e^{t\Delta}[u_0+t\P \theta_0e_3]
 -\int_0^t e^{(t-s)\Delta}(t-s)\P\nabla\cdot(\theta u)(s)\dd s\,\,e_3
 -\int_0^t e^{(t-s)\Delta}\P\nabla \cdot(u\otimes u)(s)\dd s\\
&\nabla \cdot u_0=0
\end{aligned}
\right.
\end{equation}
This system has clearly the form~\eqref{absy}, with
\begin{equation}
\label{defa}
 {\bf a}=\binom{e^{t\Delta}[u_0+t\P\theta_0e_3]}{e^{t\Delta}\theta_0}
\end{equation}
and
\begin{equation*}
\label{defbb}
{\bf B}({\bf v},{\bf \tilde v})=\binom{-B_1(u,\tilde u)+B_2(u,\tilde\theta)}{-B_3(u,\tilde\theta)}.
\end{equation*}
Here, ${\bf v}=\displaystyle\binom{u}{\theta}$, $\tilde{\bf v}=\displaystyle\binom{\tilde u}{\tilde\theta}$, and  the three bilinear operators 
$B_1\colon X\times X\to X$, next $B_2\colon X\times Y\to X$ and $B_3\colon X\times Y\to Y$,
are defined by
\begin{subequations}
\begin{align}
\label{B1}
 &B_1(u,\tilde u)=\int_0^t e^{(t-s)\Delta}\P\nabla\cdot(u\otimes \tilde u)(s)\dd s,\\
 \label{B2}
 &B_2(u,\theta)=\biggl(\int_0^t e^{(t-s)\Delta}(t-s)\P\nabla\cdot(u\theta)(s)\dd s\biggr)e_3,\\
 \label{B3}
 &B_3(u,\theta)=\int_0^t e^{(t-s)\Delta}\nabla\cdot(u\theta )(s)\dd s.
\end{align}
\end{subequations}

We will make  of use the following notations for the standard gaussian
\begin{align*}
 &G_t(x)=\frac{e^{-|x|^2/(4t)}}{(4\pi t)^{3/2}}.
\end{align*}

The starting point of our analysis will be the following theorem:

\begin{theorem}
\label{theosmall}
\hfill
\begin{enumerate}
\item
There exist two constants $\epsilon>0$ and $C>0$ such that if $\theta_0\in L^1(\R^3)$, $u_0\in L^3(\R^3)$, and
\begin{equation}
 \label{smallass}
 \|u_0\|_3+\|\theta_0\|_1<\epsilon,
\end{equation}
then there exist a unique solution $(u,\theta)\in X\times Y$ of the integral system~\eqref{IE} such that
\begin{equation*}
 \|u\|_X+\|\theta\|_Y\le C\epsilon.
\end{equation*}
\item
If $1<p<3$ and under the additional conditions $u_0\in L^p$ and $\|u_0\|_3<\epsilon_p$ for some 
$0<\epsilon_p\le \epsilon$, then the solution constructed in the previous item satisfies
\begin{equation}
\label{ubep}
\sup_{t>0} (1+t)^{\frac{1}{2}(1-\frac{3}{p})}\|u(t)\|_p <\infty.
\end{equation}
If $3\le p\le \infty$, then $\sup_{t>0}t^{\frac{1}{2}(1-\frac{3}{p})}\|u(t)\|_p <\infty$ without any additional assumption on~$u_0$.
 \end{enumerate}
\end{theorem}

\medskip
In this functional setting the proof becomes quite standard:
the first conclusion of the 
theorem relies on the following estimates, implying the continuity of the bilinear operator~${\bf B}$
on the Banach space $E$.

\begin{lemma}
\label{lemma-est}
The following estimates hold for any $u,\tilde u\in X$ and $\theta,\tilde \theta\in Y$:
\begin{equation}
\label{3est}
\begin{aligned}
&\|B_1(u,\tilde u)\|_X\le C\|u\|_X\|\tilde u\|_{X},\\
&\|B_2(u,\theta)\|_X \le C\|u\|_X\|\theta\|_Y,\\
&\|B_3(u,\tilde\theta)\|_Y \le C\|u\|_X\|\tilde \theta\|_Y.
\end{aligned}
\end{equation}
In particular, for all ${\bf v},\tilde{\bf v}\in E$, the following estimate holds:
\begin{equation}
\label{coesb}
\|{\bf B}({\bf v},\tilde{\bf v})\|_E\le C\|{\bf v}\|_E\|\tilde {\bf v}\|_E.
\end{equation}
\end{lemma}

In fact, we will establish a slightly more general version of Lemma~\ref{lemma-est}.
For this, let $1<p\le 3$ and let us define the following Banach space through the norm
\begin{equation}
\label{katop}
\|u\|_{X_p}
=\|u\|_X + \sup_{t>0} (1+t)^{\frac{1}{2}(1-\frac{3}{p})}\|u(t)\|_p.
\end{equation}
Notice that $X_3=X$ and for $1<p\le 3$ the space $X_p$ is continuously embedded in~$X$.
Next we set $E_p=X_p\times Y$.

\begin{lemma}
The first two estimates in~\eqref{3est} can be generalized as follows:
\begin{equation}
\label{2estp}
\begin{split}
&\|B_1(u,\tilde u)\|_{X_p}\le C_p\|u\|_{X_p}\|\tilde u\|_{X},\\
&\|B_2(u,\tilde\theta)\|_{X_p} \le C_p\|u\|_{X_p}\|\tilde\theta\|_Y,\\
\end{split}
\qquad
1<p\le 3.
\end{equation}
In particular, we have the estimate
\begin{equation}
\label{coesbb}
\|{\bf B}({\bf v},\tilde{\bf v})\|_{E_p}\le C_p\|{\bf v}\|_{E_p}\|\tilde {\bf v}\|_E,
\qquad 1<p\le 3.
\end{equation}
\end{lemma}

\begin{proof}
The first of estimate~\eqref{3est} is due to Kato, see, for example, 
\cite{Mey97}.
The proof of the second and of the third estimates are very similar and they are  left to the reader. 
The generalisation~\eqref{2estp} of the above estimates is straightforward.
The only thing that are needed to establish all these estimates are the 
standard H\"older and Young inequalities, and the well known fact that
the kernel $F(x,t)$ of the operator $e^{t\Delta}\P\nabla$ satisfies
\begin{equation}
 \label{kerF}
 F(x,t)=t^{-2}F(\textstyle\frac{x}{\sqrt t},1), \qquad\hbox{and}\qquad
 F(\cdot,1)\in L^1\cap L^\infty
\end{equation}
and that the kernel $K(x,t)$ of the operator $e^{t\Delta}\P$ satisfies
\begin{equation}
 \label{kerK}
 K(x,t)=t^{-3/2}K(\textstyle\frac{x}{\sqrt t},1), \qquad\hbox{and}\qquad
 K(\cdot,1)\in \bigcap_{1<p\le\infty} L^p(\R^3).
\end{equation}
Indeed, we recall the well-known pointwise estimates for these kernels (see, e.g., \cite{Miy00})
\begin{equation}
\label{poifk}
 |F(x,1)|\le C(1+|x|)^{-4},\qquad\text{and}\qquad |K(x,1)|\le C(1+|x|)^{-3}.
\end{equation}
\end{proof}

We need also the corresponding linear estimates:

\begin{lemma}
\label{lemlin}
Let $u_0\in L^3(\R^3)$ and $\theta_0\in L^1$. Let ${\bf a}$ be defined as in~\eqref{defa}.
Then, for an absolute constant~$c>0$,
\begin{equation}
\label{lines}
\|{\bf a}\|_E\le c(\|u_0\|_{3}+\|\theta_0\|_1).
\end{equation}
Moreover, if $1<p\le 3$ and we have also $u_0\in L^p(\R^3)$, then ${\bf a}\in E_p$.
\end{lemma}

The conclusion of~Lemma~\ref{lemlin} follows immediately from the usual properties of the heat kernel, and from the fact
that $\P$ is a bounded operator in $L^p$, for $1<p<\infty$.

\medskip

\begin{proof}[Proof of Theorem~\ref{theosmall}]
Estimates~\eqref{coesb}-\eqref{lines} and the standard fixed point Lemma (see, for example,
\cite{Lem02}) imply Part~1 of the theorem. The solution is obtained as the limit in~$E$ of the sequence~$({\bf v}_n)$ 
recursively defined by ${\bf v}_0={\bf a}$ and ${\bf v}_{n+1}={\bf a}+{\bf B}({\bf v}_n,{\bf v}_n)$, $n\in\N$.

Take now $1<p\le 3$.
Lemma~\ref{lemlin},  and estimate~\eqref{coesbb}, imply that the sequence of approximating solutions ${\bf v}_n$ remains bounded in~$E_p$, provided that $u_0\in L^p(\R^3)$, $\|u_0\|_3\le \epsilon_p$ and $\epsilon_p$ is small enough.
But the  balls 
$\{{\bf v}\in E_p\colon \|{\bf v}\|_{E_p}\le R\}$ are closed subsets of~$E$ (this last claim follows from Fatou's Lemma) and so the solution must belong to~$E_p$
This establishes Part~2 of the theorem.
\end{proof}

\section{Analysis of the solution in the region $|x|\ge A \sqrt t$}

The basic Young $L^p-L^q$ convolution estimates read
\begin{equation}
\label{young}
\|f*g\|_p\le \|f\|_r\|g\|_q, \qquad 1+\frac1p=\frac1r+\frac1q, \quad 1\le p,q,r\le\infty. 
\end{equation}

Here we are interested in the following variant of~\eqref{young}, that provides more information about the
behavior of $f*g(x)$ as $x\to\infty$ in the $L^p$ sense.
Notice that estimate~\eqref{young-r} below boils down (excepted for the unimportant coefficient~$2$)
to~\eqref{young} in the particular case $R=0$.

\begin{proposition}
\label{prop:yo-r}
For $R\ge0$, let $B_R=\{x\in\R^3\colon |x|<R\}$ and $B_R^c$ its complementary in~$\R^3$.
Let also $1\le p,q,r,\tilde r,\tilde q\le\infty$, such that
\[
1+\frac1p=\frac1r+\frac1q=\frac1{\tilde r}+\frac1{\tilde q}.
\]
Then,
\begin{equation}
\label{young-r}
\|f*g\|_{L^p(B_R^c)}
\le 2\Bigl(\|f\|_{L^r(B_{R/2}^c)}\|g\|_q 
     +\|f\|_{\tilde r}\|g\|_{L^{\tilde q}(B_{R/2}^c)}\Bigr).
\end{equation}
\end{proposition}

\begin{proof}
We decompose $f=f1_{B_{R/2}}+f1_{B_{R/2}^c}$ and $g=g1_{B_{R/2}}+g1_{B_{R/2}^c}$,
where $1_A$ denotes the indicator function of the set $A$.
So $f*g$ is naturally written as the sum of four terms.
As the support of $(f1_{B_{R/2}})*(g1_{B_{R/2}})$ is contained in $B_R$, 
its $L^p(B_R^c)$-norm is zero.
Applying the triangle inequality, next the classical Young inequality to the three other remaining terms,
\[
\|f*g\|_{L^p(B_R^c)} \le 
\|f\|_{L^r(B_{R/2}^c)}\bigl(\|g\|_{L^q(B_{R/2})} + \|g\|_{L^q(B_{R/2}^c)}\bigr)
+\|f\|_{\tilde r} \|g\|_{L^{\tilde q}(B_{R/2}^c)},
\]
which is even slightly stronger than~\eqref{young-r}.
\end{proof}

\begin{lemma}
\label{lem:linkato-r}
There exist positive constants $C_0$, $\eta$ and $A_0$ such that, for all $A\ge A_0$, 
\begin{equation*}
\limsup_{t\to+\infty}
 \Vert e^{t\Delta}\theta_{0} \Vert_{L^1(B^c_{A\sqrt t})} \leq C_{0}e^{-\eta A^{2}}\Vert \theta_{0}\Vert_{L^{1}},
\end{equation*}
\begin{equation*}
\limsup_{t\to+\infty}
 \|e^{t\Delta}u_0+t\,e^{t\Delta}\P \theta_0e_3\|_{L^3(B^c_{A\sqrt t})}
 \le C_0 A^{-2}\Bigl(\|u_0\|_3+\|\theta_0\|_1\Bigr)
\end{equation*}
and
\begin{equation*}
\limsup_{t\to+\infty} 
\sqrt t \|e^{t\Delta}u_0+t\,e^{t\Delta}\P\theta_0 e_3\|_{L^\infty(B_{A\sqrt t}^c)}
 \le C_0A^{-3}\Bigl(\|u_0\|_3+\|\theta_0\|_1\Bigr).
\end{equation*}
\end{lemma}

\begin{proof}

We claim that there exist two constants $C,\delta>0$ such that
\begin{equation}
\label{limsup1}
\begin{split}
& \limsup_{t\to+\infty}
 \Vert e^{t\Delta}\theta_{0} \Vert_{L^1(B^c_{A\sqrt t})} \leq Ce^{-\delta A^{2}}\Vert \theta_{0}\Vert_{L^{1}},\quad     
\limsup_{t\to+\infty}\|e^{t\Delta}u_0\|_{L^3(B^c_{A\sqrt t})}
 \le Ce^{-\delta A^2}\|u_0\|_3,\\
&\text{and} \qquad\limsup_{t\to+\infty} \sqrt t \|e^{t\Delta}u_0\|_{L^\infty(B^c_{A\sqrt t})}
 \le Ce^{-\delta A^2}\|u_0\|_3.
\end{split}
\end{equation}
To see this, apply Proposition~\ref{prop:yo-r} with $f=G_t$, $g=\theta_0$, $f=G_t$, $g=u_0$, and $(p,r,\tilde r)=(1,1,1)$,
$(p,r,\tilde r)=(3,1,1)$, 
$(p,r,\tilde r)=(\infty,\frac32,\frac32)$ respectively. Then our claim follows by 
the dominated convergence theorem.

Next, recalling the scaling and decay 
properties of the kernel~$K(\cdot,t)$ of the operator $e^{t\Delta}\P$, \eqref{kerK}--\eqref{poifk},
we see that, for $1<r\le\infty$:
\[
\|K(\cdot,t)\|_{L^r(B_{A\sqrt t}^c)}\le C (A\sqrt t)^{-3+3/r}.  
\]
 application of Proposition~\ref{prop:yo-r} with $f=K(\cdot,t)$, $g=\theta_0$ and
$(p,r,\tilde r)=(p,p,p)$, by the dominated convergence theorem
\[
\limsup_{t\to+\infty} t^{\frac12(1-3/p)}\|t\,e^{t\Delta}\P\theta_0 e_3\|_{L^p(B_{A\sqrt t}^c)}
\le C\|\theta_0\|_1  A^{-3+3/p}.  
\] 
Taking here $p=3$ and $p=\infty$ yields the assertion of the Lemma.
\end{proof}

\begin{proposition}
\label{pr:kato-r}
Let $u_0\in L^3(\R^3)$ be a divergence-free vector field, $\theta_0\in L^1(\R^n)$,
such that
\[
\|u_0\|_3+\|\theta_0\|_1<\epsilon'.
\]
If $\epsilon'>0$ is small enough ($\epsilon'$ may need to be smaller than the constant $\epsilon$ in Theorem~\ref{theosmall}),
then there exist two constants $A_0\ge1$ and $\kappa>0$ such that for all $A\ge A_0$ the solution $(u,\theta)$ obtained in Theorem~\ref{theosmall} satisfies
\begin{equation}
\label{limsup0}
\begin{split}
& \limsup_{t\to+\infty} \|\theta(t)\|_{L^1(B^c_{A\sqrt t})}\le \kappa A^{-1},\\
&\limsup_{t\to+\infty} \|u(t)\|_{L^3(B^c_{A\sqrt t})}\le \kappa A^{-2}, \qquad\text{and}\\
&\limsup_{t\to+\infty}  \sqrt t\, \|u(t)\|_{L^\infty(B_{A\sqrt t}^c)}\le \kappa A^{-3}.
\end{split}
\end{equation}
\end{proposition}

\begin{proof}
Recalling the properties~\eqref{kerF} and \eqref{poifk}
of the kernel $F(x,t)$ of the operator
$e^{t\Delta}\P\nabla$, we obtain the two estimates:
\begin{subequations}
\begin{equation}
\label{est-f}
\|F(t-s)\|_r\le C(t-s)^{-2+3/(2r)},
\end{equation}
and
\begin{equation}
\label{est-f-r}
\begin{split}
\|F(t-s)\|_{L^r(B^c_{A\sqrt t})}
&= (t-s)^{-2+3/(2r)}\Bigl(\int_{|x|\ge A\sqrt t/\sqrt{t-s}}|F(x,1)|^r\dd x\Bigr)^{1/r}\\
&\le CA^{-4+3/r}t^{-2+3/(2r)}.
\end{split}
\end{equation}
\end{subequations}

Consider the sequence of approximate solutions $(u_n,\theta_n)$ $(n\ge1)$, where
\[ 
u_1=e^{t\Delta}u_0+t\,e^{t\Delta}\P \theta_0 e_3, \qquad \theta_1=e^{t\Delta}\theta_0
\] 
and
\begin{equation}
\label{recur}
u_{n+1}=u_1-B_1(u_n,u_n)-B_2(u_n,\theta_n), \qquad \theta_{n+1}=\theta_1-B_3(u_n,\theta_n).
\end{equation}
By the usual fixed point argument the sequence, $(u_n,\theta_n)$ converges to the solution $(u,\theta)$ obtained in Theorem~\ref{theosmall} in the
$(X\times Y)$-norm.
Moreover, there is an absolute constant $C>0$ such that, for all $n\ge1$,
\[
\sup_{n\ge1}(\|u_n\|_X+\|\theta_n\|_Y)\le C\epsilon\equiv\varepsilon,
\]
and so, in particular,
$\|u\|_X+\|\theta\|_Y\le\varepsilon$.

We first need to prove that, for all $n\ge 1$ and some constant $\kappa_n>0$, we have
\begin{subequations}
\begin{equation}
\label{se0}
\limsup_{t\to+\infty} \|\theta_n(t)\|_{L^1(B_{A\sqrt t}^c)} \le \kappa_n A^{-1},
\end{equation}
\begin{equation}
\label{se1}
\limsup_{t\to+\infty} \|u_n(t)\|_{L^3(B_{A\sqrt t}^c)} \le \kappa_n A^{-2}
\end{equation}
and
\begin{equation}
\label{se2}
\limsup_{t\to+\infty} \sqrt t \|u_n(t)\|_{L^\infty(B_{A\sqrt t}^c)} \le \kappa_n A^{-3}.
\end{equation}
\end{subequations}
Let us proceed by induction.
For $n=1$,  \eqref{se0}--\eqref{se2} hold true for $A_{0}$ big enough because of Lemma~\ref{lem:linkato-r}.
Assume now \eqref{se0}--\eqref{se2} hold at the step~$n$ and let us prove their validity at the step
$n+1$.
Let $t_A\ge1$ (with $t_A$ possibly depending also on~$n$) such that for all $t\geq t_{A}$,
\begin{equation}
\|\theta_n(t)\|_{L^1(B_{A\sqrt t}^c)} \le 2\kappa_n A^{-1}, \quad
\|u_n(t)\|_{L^3(B_{A\sqrt t}^c)} \le 2\kappa_n A^{-2} 
\quad\text{and}\quad
\sqrt t\| u_n(t)\|_{L^\infty(B_{A\sqrt t}^c)} \le 2\kappa_n A^{-3}
\end{equation}
We now write
\[
B_3(u_n,\theta_n)(4t)=\biggl(\int_0^{t_A}+\int_{t_A}^{4t}\biggr) \tilde{F}(4t-s)*(u_n \theta_n)(s)\dd s,
\]
\[
B_1(u_n,u_n)(4t)=\biggl(\int_0^{t_A}+\int_{t_A}^{4t}\biggr) F(4t-s)*(u_n\otimes u_n)(s)\dd s
\]
and
\[
B_2(u_n,\theta_n)(4t)=\biggl(\int_0^{t_A}+\int_{t_A}^{4t}\biggr) (4t-s)F(4t-s)*(u_n \theta_n)(s)\dd s
\]
The reason for considering the time $4t$ instead of $t$ will be clear in~\eqref{B1-r} below. 
The kernel $\tilde{F}(x,t)$ of the operator $e^{t\Delta}\nabla$ has similar properties as the kernel $F(x,t)$ and satisfies
\begin{equation}
 \label{kerFF}
 \tilde{F}(x,t)=t^{-2}\tilde{F}(\textstyle\frac{x}{\sqrt t},1), \qquad\hbox{and}\qquad
 \tilde{F}(\cdot,1)\in L^1\cap L^\infty
\end{equation}
Moreover,
\begin{subequations}
\begin{equation}
\label{est-ff}
\|\tilde{F}(t-s)\|_r\le C(t-s)^{-2+3/(2r)},
\end{equation}
and
\begin{equation}
\label{est-ff-r}
\begin{split}
\|\tilde{F}(t-s)\|_{L^r(B^c_{A\sqrt t})}
&= (t-s)^{-2+3/(2r)}\Bigl(\int_{|x|\ge A\sqrt t/\sqrt{t-s}}|\tilde{F}(x,1)|^r\dd x\Bigr)^{1/r}\\
&\le CA^{-4+3/r}t^{-2+3/(2r)}.
\end{split}
\end{equation}
\end{subequations}
Only the three integrals $\int_{t_A}^{4t}$ will play a role as $t\to\infty$. Indeed, for $t>t_A$,
\[
\begin{split}
\biggl\| \int_0^{t_A} F(4t-s)*(u_n\otimes u_n)(s)\biggr\|_3
&\le C\int_0^{t_A}\|F(4t-s)\|_{3/2}\|u_n(s)\|_3^2\dd s\\
&\le C(4t-t_A)^{-1} t_A \|u_n\|_X^2\le C\varepsilon^2 (4t-t_A)^{-1} t_A .
\end{split}
\]
Moreover,
\[
\begin{split}
\sqrt{4t}\biggl\| \int_0^{t_A} F(4t-s)*(u_n\otimes u_n)(s)\biggr\|_\infty
&\le C\sqrt t\int_0^{t_A}\|F(4t-s)\|_{3}\|u_n(s)\|_3^2\dd s\\
&\le C\sqrt t(4t-t_A)^{-3/2} t_A \|u_n\|_X^2\le C\varepsilon^2 \sqrt t(t-t_A)^{-3/2} t_A .
\end{split}
\]
Taking the the limit as $t\to+\infty$ in the above expressions we get
\begin{equation}
\label{lim01}
\lim_{t\to+\infty}  \biggl\|\int_0^{t_A} F(4t-s)*(u_n\otimes u_n)(s)\dd s\biggr\|_3
= \lim_{t\to+\infty}
  \sqrt{4t} \biggl\|\int_0^{t_A} F(4t-s)*(u_n\otimes u_n)(s)\dd s\biggr\|_\infty=0.
\end{equation}

In a similar way (using now $\|\theta_n(s)\|_1\le \varepsilon$ and, respectively, $\|u_n(s)\|_3\le \varepsilon$ or $\sqrt s\|u_n(s)\|_\infty\le\varepsilon$), we can prove that
\begin{equation}
\label{lim02}
\lim_{t\to+\infty}  \biggl\|\int_0^{t_A} (4t-s)F(4t-s)*(u_n\theta_n)(s)\dd s\biggr\|_3
=\lim_{t\to+\infty}
 \sqrt{4t} \biggl\|\int_0^{t_A} (4t-s)F(4t-s)*(u_n \theta_n)(s)\dd s\biggr\|_\infty=0.
\end{equation}
Furthermore,
\begin{eqnarray}\nonumber
\left\Vert \int_{0}^{t_{A}} \tilde{F}(4t-s)*(u_{n}\theta_{n})(s)\dd s \right\Vert_{1} &\leq&  \int_{0}^{t_{A}} \Vert \tilde{F}(4t-s)\Vert_{1}\Vert u_{n}(s)\Vert_{\infty}\Vert \theta_{n}(s)\Vert_{1} \dd s \\ \nonumber 
&\leq& C\varepsilon^{2} t_{A}^{3/4} \int_{0}^{4t} (4t-s)^{-1/2}s^{-3/4}\dd s \leq C\varepsilon^{2} t_{A}^{3/4} t^{-1/4}
\end{eqnarray}
Hence,
\begin{equation}
\label{lim03}
\lim_{t\to+\infty}  \biggl\|\int_0^{t_A} \tilde{F}(4t-s)*(u_{n}\theta_{n})(s) \dd s\biggr\|_1
=0.
\end{equation}

Applying Proposition~\ref{prop:yo-r} with $p=3$, $r=3/2$ and $\tilde r=1$,
we get from~\eqref{est-f}--\eqref{est-f-r} and our inductive assumption
\begin{equation}
\label{B1-r}
\begin{split}
&\biggl\|\int_{t_A}^{4t} F(4t-s)*(u_n\otimes u_n)(s)\biggr\|_{L^3(B_{A\sqrt{4t}}^c)} \\
&\qquad \le C\int_{t_A}^{4t} \Bigl[A^{-2}t^{-1}\|u_n(s)\|_3\|u_n(s)\|_3
 	+ (4t-s)^{-1/2}\|u_n(s)\|_{L^3(B_{\frac{A}{2}\sqrt{4t}}^c)}\|u_n(s)\|_\infty\Bigr]\dd s\\
&\qquad\le CA^{-2}\int_0^{4t} \Bigl[t^{-1}\varepsilon^2
		+ \kappa_n \varepsilon(4t-s)^{-1/2}s^{-1/2}\Bigr]\dd s\\
&\qquad\le CA^{-2}\varepsilon(\varepsilon+\kappa_n).
\end{split}
\end{equation}
In the same way, applying now Proposition~\ref{prop:yo-r} with $p=3$, $r=3$ and $\tilde r=3$,
\begin{equation}
\begin{split}
&\biggl\|\int_{t_A}^{4t} (4t-s)F(4t-s)*(u_n \theta_n)(s)\biggr\|_{L^3(B_{A\sqrt{4t}}^c)} \\
&\qquad \le C\int_{t_A}^{4t} \Bigl[A^{-2}t^{-3/2}(4t-s)\|u_n(s)\|_{3}\|\theta_n(s)\|_{3/2}
 	+ (4t-s)^{-1/2}\|u_n(s)\|_{L^3(B_{\frac{A}{2}\sqrt{4t}}^c)}
 	    \|\theta_n(s)\|_{3/2}
 	   \dd s\\
&\qquad\le C A^{-2}\int_1^{4t} \Bigl[t^{-3/2}(4t-s)s^{-1/2}\varepsilon^2
  +\kappa_n\varepsilon(4t-s)^{-1/2}s^{-1/2}
		\Bigr]\dd s\\
&\qquad\le CA^{-2}\varepsilon(\varepsilon+\kappa_n).
\end{split}
\end{equation}
Summarizing, recalling the relation between $u_{n+1}$ and $(u_n,\theta_n)$, we proved so far that
\begin{equation}
\label{sen+1}
\begin{split}
A^2\limsup_{t\to+\infty} \|u_{n+1}(t)\|_{L^3(B_{A\sqrt t}^c)} 
&=A^2\limsup_{t\to+\infty} \|u_{n+1}(4t)\|_{L^3(B_{A\sqrt{4t}}^c)}\\ 
&\le	C_0\Bigl(\|u_0\|_3+\|\theta_0\|_1\Bigr)+2C\varepsilon(\varepsilon+\kappa_n).
\end{split}
\end{equation}

On the other hand, applying once more Proposition~\ref{prop:yo-r} with $p=\infty$, $r=3$ and $\tilde r=6/5$, and our inductive assumption,
we get from~\eqref{est-f}--\eqref{est-f-r},
\begin{equation}
\label{B1-ri}
\begin{split}
&\sqrt{4t} \biggl\|\int_{t_A}^{4t} F(4t-s)*(u_n\otimes u_n)(s)\biggr\|_{L^\infty(B_{A\sqrt{4t}}^c)} \\
&\qquad \le C\sqrt t\int_{t_A}^{4t} \Bigl[A^{-3}t^{-3/2}\|u_n(s)\|_3\|u_n(s)\|_3
 	+ (4t-s)^{-3/4}\|u_n(s)\|_{L^\infty(B_{\frac{A}{2}\sqrt{4t}}^c)}\|u_n(s)\|_6\Bigr]\dd s\\
&\qquad\le CA^{-3}\sqrt t\int_1^{4t} \Bigl[t^{-3/2}s^{-1/2}\varepsilon^2
		+ \kappa_n \varepsilon(4t-s)^{-3/4}s^{-3/4}\Bigr]\dd s\\
&\qquad\le CA^{-3}\varepsilon(\varepsilon+\kappa_n).
\end{split}
\end{equation}
Moreover, applying now Proposition~\ref{prop:yo-r} with $p=\infty$, $r=\infty$ and $\tilde r=6$,
\begin{equation}
\begin{split}
&\sqrt{4t}\biggl\|\int_{t_A}^{4t} 
   (4t-s)F(4t-s)*(u_n \theta_n)(s)\biggr\|_{L^\infty(B_{A\sqrt{4t}}^c)} \\
&\qquad \le C\sqrt t\int_{t_A}^{4t} \Bigl[A^{-4}t^{-2}(4t-s)\|u_n(s)\|_{\infty}\|\theta_n(s)\|_{1}
 	+ (4t-s)^{-3/4}\|u_n(s)\|_{L^\infty(B_{\frac{A}{2}\sqrt{4t}}^c)}  \|\theta_n(s)\|_{6/5}
 	   \dd s\\
&\qquad\le C A^{-3}\sqrt t\int_1^{4t} \Bigl[A^{-1}t^{-2}(4t-s)s^{-1/2}\varepsilon^2
 		+\kappa_n\varepsilon(4t-s)^{-3/4}s^{-3/4}
		\Bigr]\dd s\\
&\qquad\le CA^{-3}\varepsilon(\varepsilon+\kappa_n).
\end{split}
\end{equation}
These estimates imply
\begin{equation}
\label{sen+2}
\begin{split}
A^3\limsup_{t\to+\infty} \sqrt t\|u_{n+1}(t)\|_{L^\infty(B_{A\sqrt t}^c)} 
&=A^3\limsup_{t\to+\infty} \sqrt {4t}\|u_{n+1}(4t)\|_{L^\infty(B_{A\sqrt{4t}}^c)}\\
&\le	C_0\Bigl(\|u_0\|_3+\|\theta_0\|_1\Bigr)+2C\varepsilon(\varepsilon+\kappa_n).
\end{split}
\end{equation}

 Moreover, applying Proposition~\ref{prop:yo-r} with $p=1$, $r=1$ and $\tilde r=1$,
we get from~\eqref{est-ff}--\eqref{est-ff-r} and our inductive assumption
\begin{equation}
\begin{split}
&\biggl\|\int_{t_A}^{4t} 
   \tilde{F}(4t-s)*(u_n \theta_n)(s)\biggr\|_{L^1(B_{A\sqrt{4t}}^c)} \\
&\qquad \le C\int_{t_A}^{4t} \Bigl[A^{-1}t^{-1/2}\|u_n(s)\|_{\infty}\|\theta_n(s)\|_{1}
 	+ (4t-s)^{-1/2}\|\theta_n(s)\|_{L^1(B_{\frac{A}{2}\sqrt{4t}}^c)}  \|u_n(s)\|_{\infty}
 	   \dd s\\
&\qquad\le C A^{-1}\int_0^{4t} \Bigl[t^{-1/2}s^{-1/2}\varepsilon^2
 		+\kappa_n\varepsilon(4t-s)^{-1/2}s^{-1/2}
		\Bigr]\dd s\\
&\qquad\le CA^{-1}\varepsilon(\varepsilon+\kappa_n).
\end{split}
\end{equation}
Hence, we get
\begin{equation}
\label{sen+2-ff}
\begin{split}
A\limsup_{t\to+\infty} \|\theta_{n+1}(t)\|_{L^1(B_{A\sqrt t}^c)} 
&=A\limsup_{t\to+\infty} \|\theta_{n+1}(4t)\|_{L^1(B_{A\sqrt{4t}}^c)}\\
&\le	C_0\Bigl(\|u_0\|_3+\|\theta_0\|_1\Bigr)+2C\varepsilon(\varepsilon+\kappa_n).
\end{split}
\end{equation}
Combining~\eqref{sen+1}--\eqref{sen+2-ff}, 
we deduce that~\eqref{se1}--\eqref{se2} hold true at the step $n+1$ with a constant 
\[
\kappa_{n+1}\le C_0 \Bigl(\|u_0\|_3+\|\theta_0\|_1\Bigr)+2C\varepsilon^2+2C\varepsilon \kappa_n,
\]
where $C>0$ is an absolute constant.
Our smallness assumption on the initial data ensures $2C\varepsilon<1$,
so that one gets $\kappa\equiv\sup_{n\ge1} \kappa_n<\infty$.
We thus deduce from~\eqref{se0}--\eqref{se2} that
\begin{equation}
\label{suu1}
\begin{split}
&\limsup_{t\to+\infty} \|\theta(t)\|_{L^1(B_{A\sqrt t}^c)} \le \kappa A^{-1} \\
&\limsup_{t\to+\infty} \|u(t)\|_{L^3(B_{A\sqrt t}^c)} \le \kappa A^{-2}\qquad\text{and} \\
 &\limsup_{t\to+\infty}  \sqrt t \|u(t)\|_{L^\infty(B_{A\sqrt t}^c)} \le \kappa A^{-3}.
\end{split}
\end{equation}
\end{proof}

\section{Large time growth of $L^p$-norm, $1<p<3$}
\label{sec:growth}

\subsection{Statement of the main result}
The goal of this section is to establish lower bound estimates for $\|u(t)\|_{p}$ that precisely agree with
the corresponding upper bounds obtained in Theorem~\ref{theosmall} (see~\eqref{ubep}).
Namely, we establish the following.

\begin{theorem}
\label{th:lube}
Let $1<p \leq\infty$. Let $(u,\theta)$ be the solution constructed in Theorem~\ref{theosmall}.
We put the additional conditions on the data $\int|x|^q\min\{1,|\theta_0(x)|^q\}\dd x<\infty$ 
for some $1\le q<3/2$ such that $q\le p$.
When $1<p<3$, we also assume $u_0\in L^p(\R^3)$.
Then there exist two constants $c_1,c_2>0$ and $t_0>0$ such that for all $t\ge t_0$,
\begin{equation}
 \label{lube}
 c_1\bigl|\textstyle\int\theta_0\bigr|\,t^{-\frac12(1-3/p)}\le \|u(t)\|_p\le c_2\,t^{-\frac12(1-3/p)}.
\end{equation}
In particular, for $3< p\le \infty$, $\|u(t)\|_p\to0$. If otherwise $1<p<3$ and $\int\theta_0\not=0$, 
then $\|u(t)\|_p\to+\infty$.
\end{theorem}

The subtle point in establishing the lower bounds is that all the terms contributing to $u(t)$, \emph{i.e.},
the three terms on the right hand side of the equality,
\begin{equation}
 \label{eq:u=}
 u(t)=e^{t\Delta}[u_0+t\P\theta_0e_3]-B_1(u,u)-B_2(u,\theta),
\end{equation}
individually have the same behavior $\simeq t^{-\frac12(1-3/p)}$ as~$t\to+\infty$ in the $L^p$-norm.
The strategy will consist in proving that these three terms cannot compensate each other. To achieve this, we will compute their $L^p$-norm in regions $\{x\colon |x|\ge A\sqrt t\}$, finding a behavior of the form 
$C_j(A)\,t^{-\frac12(1-3/p)}$
as $t\to+\infty$ for each of them ($j=1,2,3$). But for large enough $A$, the constant $C_1(A)$
corresponding to the first term $e^{t\Delta}[u_0+t\P\theta_0e_3]$ turns
out to be much larger than the corresponding constants $C_2(A)$ and $C_3(A)$
of $B_1(u,u)$ and $B_2(u,\theta)$ respectively.
This implies that $u(t)$ can be bounded from below in the 
$L^p$-norm on $\{x\colon |x|\ge A\sqrt t\}$, and so in the $L^p(\R^n)$ norm as announced by
\eqref{lube}.

The result of Theorem~\ref{th:lube} is in agreement with that of~\cite{BraS12}, where
estimates~ of the form~\eqref{lube} were obtained under much more stringent conditions on the
data (including pointwise decay assumptions on $u_0$, $\theta_0$ and somewhat artificial
smallness smallness conditions on $\theta_0$ and $u_0$, etc.).

The condition
$$\int|x|^q\min\{1,|\theta_0(x)|^q\}\dd x<\infty$$ in Theorem \ref{th:lube} could be seen as a weaker integral version of the pointwise estimates in \cite{BraS12} on data. We do not know if we can remove it.

\subsection{$L^p$-analysis of the bilinear terms in $B_{A\sqrt t}^c$\,}
 
\begin{lemma}
\label{lem:b12uu}
Let $1<p\le \infty$.
Let $(u,\theta)$ be the solution constructed in Proposition~\ref{pr:kato-r}
 and the second part of Theorem~\ref{theosmall}, under the usual smallness assumptions on $\|u_0\|_3$ and
 $\|\theta_0\|_1$. If $1<p<3$, we require the additional condition $u_0\in L^p(\R^n)$.
 There is $A_0\ge1$ and $C>0$ such that, for all $A\ge A_0$,
\begin{equation}
 \label{b2uu}
 \limsup_{t\to+\infty} t^{\frac12(1-\frac3p)} \Bigl( \|B_1(u,u)(t)\|_{L^p(B_{A\sqrt{t}}^c)}
   +\|B_2(u,\theta)(t)\|_{L^p(B_{A\sqrt{t}}^c)}\Bigr)
   \le CA^{-4+3/p}.
\end{equation}
\end{lemma}

\begin{proof}
Let $t_A\ge1$ such that
\begin{equation}
\|u(t)\|_{L^3(B_{A\sqrt t}^c)} \le 2\kappa A^{-2} 
\quad\text{and}\quad
\sqrt t\| u(t)\|_{L^\infty(B_{A\sqrt t}^c)} \le 2\kappa A^{-3}, 
\qquad\text{for all $t\ge t_A$}.
\end{equation}
{By interpolation, we have for all $t\geq t_{A}$, 
$$
\forall r \geq 3,\quad \quad \Vert u(t)\Vert_{L^{r}(B_{A\sqrt{t}}^c)} \leq 2\kappa A^{-3+\frac{3}{r}}.
$$
}
We have, by the application of Proposition~\ref{prop:yo-r}, as $t\to+\infty$
(we assume $3\leq p<\infty$ in the calculations below; if $p=\infty$ the result remains true
with slight changes in the choice of the exponents in Proposition~\ref{prop:yo-r}):
\begin{equation}
\begin{split}
 &(4t)^{\frac12(1-\frac3p)}\|B_1(u,u)(4t)\|_{L^p(B_{A\sqrt{4t}}^c)}\\
  &\quad\le  Ct^{\frac12(1-\frac3p)}
   \int_{t_A}^{4t} A^{-4+\frac3p}t^{-2+\frac{3}{2p}}\|u(s)\|_p\|u(s)\|_{p'}
 	+(4t-s)^{-\frac12}\|u(s)\|_{L^{p}(B_{\frac{A}{2}\sqrt{4t}}^c)}\|u(s)\|_{{L^\infty(B_{\frac{A}{2}\sqrt{4t}}^c)}}\dd s + o(1)\\
 &\quad\le  Ct^{\frac12(1-\frac3p)}\int_0^{4t} A^{-4+3/p}t^{-2+3/(2p)} s^{1/2}
 	+(4t-s)^{-1/2} A^{-6+3/p} s^{-1+3/(2p)}\dd s + o(1)\\
 &\quad\le CA^{-4+3/p} + o(1).
\end{split}
\end{equation}
The same holds for $1<p<3$,
\begin{equation}
\begin{split}
 &(4t)^{\frac12(1-\frac3p)}\|B_1(u,u)(4t)\|_{L^p(B_{A\sqrt{4t}}^c)}\\
  &\quad\le  Ct^{\frac12(1-\frac3p)}
   \int_{t_A}^{4t} A^{-4+\frac3p}t^{-2+\frac{3}{2p}}\|u(s)\|_p\|u(s)\|_{p'}
 	+(4t-s)^{-\frac12}\|u(s)\|_p\|u(s)\|_{{L^\infty(B_{\frac{A}{2}\sqrt{4t}}^c)}}\dd s + o(1)\\
 &\quad\le  Ct^{\frac12(1-\frac3p)}\int_0^{4t} A^{-4+3/p}t^{-2+3/(2p)} s^{1/2}
 	+(4t-s)^{-1/2} A^{-3} s^{-1+3/(2p)}\dd s + o(1)\\
 &\quad\le CA^{-4+3/p} + o(1).
\end{split}
\end{equation}
Here, and below, the $o(1)$ arise from the contribution of the integral $\int_0^{t_A}$.

Similarly, as $t\to+\infty$,
\begin{equation}
\begin{split}
 &(4t)^{\frac12(1-\frac3p)}\|B_2(u,\theta)(4t)\|_{L^p(B_{A\sqrt{4t}}^c)}\\
  &\quad\le  Ct^{\frac12(1-\frac3p)}
   \int_{t_A}^{4t}\biggl[ A^{-4+\frac3p}t^{-2+\frac{3}{2p}}(4t-s)\|u(s)\|_\infty\|\theta(s)\|_1\\
 &\qquad\qquad\qquad\qquad\qquad+(4t-s)^{-1+\frac{3}{2p}}\|u(s)\|_{L^\infty(B_{\frac{A}2\sqrt{4t}}^c)}\|\theta(s)\|_{L^1(B_{\frac{A}2\sqrt{4t}}^c)}\biggr]
 \dd s + o(1)\\
 &\quad\le  Ct^{\frac12(1-\frac3p)}\int_0^{4t} A^{-4+3/p}t^{-2+3/(2p)} (4t-s)s^{-1/2}
 	+(4t-s)^{-1+\frac{3}{2p}} A^{-4} s^{-1/2}\dd s + o(1)\\
 &\quad\le CA^{-4+3/p} + o(1).
\end{split}
\end{equation}
Summing~ the two last estimates and taking the $\limsup_{t\to+\infty}$ leads to~\eqref{b2uu}.

\end{proof}

\subsection{$L^p$-analysis of the linear term in $B_{A\sqrt t}^c$\,}

\begin{proposition}
\label{pro:dirac}
Let $f\in L^1(\R^n)$. Then the function
 $V(x)=-x\int_0^1 f\bigl(\frac x\lambda\bigr)\frac{\dd \lambda}{\lambda^{n+1}}$ 
belongs to $L^1_{\rm loc}(\R^n)$ and the following identity holds in the distributional sense:
\begin{equation}
\label{diracdec}
f=\Bigl(\int f\Bigr)\delta+\div V,
\end{equation}
where $\delta$ denotes the Dirac mass at zero. 
\end{proposition}

\begin{remark}
Proposition~\ref{pro:dirac} should be compared with a result of J.~Duoandikoetxea and E.~Zuazua~\cite{DuoZ92}, 
where such decomposition of~$f$
was established assuming $f\in L^1(\R^n)$ and $|x|f\in L^q(\R^n)$ for some 
$1\le q<n/(n-1)$. In this case, the authors observed the validity of the estimate
\begin{equation}
\label{zua}
 \|V\|_q\le C\|\,|x|f\|_q, \qquad 1\le q<n/(n-1),
\end{equation}
where $C$ is a constant depending only on $q$ and $n$.
\end{remark}
Estimate~\eqref{zua} readily follows applying Minkowski integral inequality to the $L^q$-norm
in the definition of~$V$. The restriction $1\le q<n/(n-1)$ is needed for the convergence of
the $L^q$-Bochner integral.

\begin{proof}
The fact that $V\in L^1_{\rm loc}(\R^n)$ follows from an elementary calculation: for $R\ge0$,
\[
\begin{split}
 \int_{|x|\le  R}|V(x)|\dd x
 &\le \int_0^1\!\!\int_{|x|\le R/\lambda}|x| \,|f(x)|\dd x\dd \lambda\\
&=\int_{|x|\le R}\int_0^1 |x| \,|f(x)|\dd \lambda \dd x
  +\int_{|x|\ge R}\int_0^{R/|x|} |x| \,|f(x)|\dd \lambda\dd x\\
 &\le R\int|f(x)|\dd x.  
\end{split}
\]
Moreover, for any test function $\varphi\in \mathcal{D}(\R^n)$,
\[
\begin{split}
 \langle f-\Bigl(\int f\Bigr)\delta,\varphi\rangle
 &=\int f(x)[\varphi(x)-\varphi(0)]\dd x
   =\int f(x)\int_0^1 x\cdot\nabla\varphi(\lambda x)\dd \lambda\dd x\\
 &=\int\!\!\!\int_0^1 f(x/\lambda)\frac{\dd \lambda}{\lambda^{n+1}} x\cdot\nabla \varphi(x)\dd x =
   \langle \div V,\varphi\rangle.   
\end{split}
\]
\end{proof}

The decomposition~\eqref{diracdec} can be applied to $\theta_0$, and allows to obtain the following:

\begin{lemma}
\label{lem:declin}
If $\theta_0\in L^1(\R^3)$, if $V$ is associated to~$\theta_0$ 
as in~Proposition~\ref{pro:dirac} and $K(\cdot,t)$ and $F(\cdot,t)$ denote respectively the kernels of 
$e^{t\Delta}\P$ and $e^{t\Delta}\P\div$, then for all $t>0$
 the convolution integral $F(\cdot,t)*V(x)=\int F(x-y,t)V(y)\dd y$
is well defined for a.e. $x\in \R^n$ and defines a locally integrable function.
Moreover,
for $j=1,2,3$,
\begin{equation}
\label{declin}
 \bigl[e^{t\Delta}\P\theta_0e_3\bigr]_{j}=\Bigl(\int\theta_0\Bigr)K_{j,3}(\cdot,t)+\sum_{h=1}^3F_{j,h,3}(\cdot,t)*V_h.
\end{equation}
\end{lemma}

\begin{proof}
 To establish that the convolution integral $F(\cdot,t)*V(x)$
 defines a locally integrable function for all $t>0$,  by~\eqref{kerF} and~\eqref{poifk}
 it is sufficient to prove that, for all $R>0$,
 the integral $\int_{|x|\le R} \int (1+|x-y|)^{-4}|V(y)|\dd y$ converges. Indeed,
 \[
 \begin{split}
\int_{|x|\le R}&\int (1+|x-y|)^{-4}|V(y)|\dd y\dd x\\
&\le C\int_{|x|\le R}\int\!\!\!\int_0^1 (1+|x-\lambda y|)^{-4}|y|\,|\theta_0(y)|\dd\lambda\dd y\dd x\\
&\le CR^3\int_0^1\int_{|y|\le 2R/\lambda}|y|\,|\theta_0(y)|\dd y\dd\lambda
   +CR^3 \int_0^1\int_{|y|\ge 2R/\lambda} \lambda^{-4}|y|^{-3}\,|\theta_0(y)|\dd y\dd\lambda\\
&\le CR^3\int\biggl(\int_0^{2R/|y|}\dd \lambda\biggr) |y|\,|\theta_0(y)|\dd y
   +CR^3\int_{|y|\ge 2R}\biggl(\int_{\lambda\ge2R/|y|}\lambda^{-4}\dd \lambda \biggr)|y|^{-3}\,|\theta_0(y)|\dd y\\
&\le CR^4\int|\theta_0(y)|\dd y
   +C\int_{|y|\ge 2R}\,|\theta_0(y)|\dd y<\infty\\
\end{split}
 \]
Owing to  decomposition~\eqref{diracdec}, 
$[K(\cdot,t)*\theta_0e_3]_j=(\int\theta_0)(K_{j,3}(\cdot,t))*\delta+ \div [K_{j,3}(\cdot,t))* V]$.
The conclusion of the lemma follows.
\end{proof}

\begin{lemma}
 \label{lem:fv-r}
 Let $1\le p\le \infty$. Assume that $\theta_0\in L^1(\R^n)$ and $\int |x|^q\,|\theta_0(x)|^q\dd x<\infty$ for some $1\le q<3/2$, with $q\le p$. Let $V$ be related to $\theta_0$ as before.
 Then, for all $A>0$ and all $t>0$, we have
\begin{equation}
 \label{fv-r}
 \lim_{t\to\infty} t^{\frac32(1-1/p)} \| F(\cdot,t)*V\|_p=0,
 \qquad 1\le p\le\infty.
\end{equation}
\end{lemma}

\begin{proof}
Recalling the estimate $\int|V|^q\le \int |x|^q|\theta_0|^q$, valid for $1\le q<3/2$  (see~\eqref{zua}),
one just needs to apply Young convolution inequality to deduce that, for $1+\frac1p=\frac1r+\frac1q$,
\[ 
\| F(\cdot,t)*V\|_p\le C\|F(t)\|_r\le Ct^{-2+\frac32(1+\frac1p-\frac1q)}.
\]
As $1\le q<3/2$, this implies the assertion of the Lemma.
\end{proof}

\begin{lemma}
 \label{lem:mainte}
Let $1<p\le \infty$, 
The kernel $K(\cdot,t)$ of the operator $e^{t\Delta}\P$. There exist two positive constants $c,A_0>0$
such that for all $A\ge A_0$,
\begin{equation}
 \label{mainte}
\begin{split}
& t^{\frac{3}{2}(1-1/p)}\|K(t)\|_{L^p(B_{A\sqrt t}^c)}\ge cA^{-3+3/p}.
\end{split}
 \end{equation} 
\end{lemma}

\begin{proof}
We know from~\cite{Bra09} that the kernel $K(\cdot,t)$ can be decomposed as
$K(x,t)=\mathfrak{K}(x)+|x|^{-3}\Psi(x/\sqrt{t})$, where the components of $\mathfrak{K}(x)$ are homogeneous function of degree $-3$ (given by second-order derivatives of the fundamental solution of the Laplacian in~$\R^3$),
and $\lim_{x\to\infty}\Psi(x)=0$ (in fact $\Psi$ decays to zero exponentially fast).
\[
\begin{split}
 \|K(t)\|_{L^p(B_{A\sqrt t}^c)}
 &=\biggl(\int_{B_{A\sqrt t}^c} \Bigl| \mathfrak{K}(x)+|x|^{-3}\Psi(x/\sqrt t)\Bigr|^p\dd x\biggr)^{1/p}\\
 &\ge t^{-\frac32(1-1/p)} A^{-3+3/p}
    \biggl|\biggl(\int_{|x|\ge1} | \mathfrak{K}(x)|^p\dd x\biggr)^{1/p}
     -\biggl(\int_{|x|\ge1}\Bigl||x|^{-3}\Psi(Ax)\Bigr|^p\dd x 
    \biggr)^{1/p}\biggr|\\
 &\ge c' t^{-\frac32(1-1/p)} A^{-3+3/p}
\end{split}
\]
where $c'>0$ if $A$ is large enough, because of the decay of $\Psi$.
\end{proof}

\begin{lemma}
\label{lem:hu}
 Let $1<p\le\infty$ and $u_0\in L^3(\R^3)$. Assume also $u_0\in L^p(\R^3)$ if $1<p<3$.
 Then
 \[
 \lim_{t\to+\infty} t^{\frac12(1-3/p)}\|e^{t\Delta}u_0\|_p=0.
 \]
\end{lemma}

\begin{proof}
If $1< p<3$, the conclusion just follows from the inequality $\|e^{t\Delta}u_0\|_p\le\|u_0\|_p$.
If $3\le p\le\infty$, this is well known: one approaches in the $L^3$-norm,
$u_0$ by a sequence of functions 
in $L^1\cap L^3$, with $L^3$ norm not exceeding $\|u_0\|_3$ and next applies 
the usual $L^3$-$L^p$ and $L^1$-$L^p$
heat estimates. In fact, let $\psi \in L^{1}\cap L^{p}$ such that
$
\Vert \psi - u_{0} \Vert_{p}\leq \varepsilon.
$
Hence, 
$$
t^{\frac{1}{2}-\frac{3}{2p}}\Vert e^{t\Delta}u_{0}\Vert_{p} \leq t^{\frac{1}{2}-\frac{3}{2p}}\varepsilon+t^{-1}\Vert \psi\Vert_{L^{1}},
$$
which proves the claimed result.
\end{proof}

\begin{proof}[Proof of Theorem~\ref{th:lube}]
By our assumptions, the data satisfy  the additional conditions $\int|x|^q\min\{1,|\theta_0(x)|^q\}\dd x<\infty$ 
for some $1\le q<3/2$ such that $q\le p$. Moreover, when $1<p<3$, $u_0\in L^p(\R^3)$.
The solution $(u,\theta)$, is such that $\sup_{t>0}\|\theta(t)\|_1<\infty$
and $\|\theta(t)\|_\infty\to0$. In particular, after some time $t_q>0$,  $\|\theta(t)\|_q\le 1$ for all $t\ge t_q$.
By a time-translation argument we can consider $(u(t_q),\theta(t_q))$ as our new initial datum.
Notice that $\int \theta_0=\int\theta(t_q)$ because the spatial mean of the temperature is preserved by the Boussinesq flow. 
So, without loss of generality,
we can work under the seemingly stronger condition $\int|x|^q|\theta_0|^q(x)\dd x<\infty$.
This observation will be useful later on, when we will apply Lemma~\ref{lem:fv-r}, where
such stronger condition was needed.

Let us apply Lemma~\ref{lem:declin} and write the velocity field as
\begin{equation}
 \label{idenu}
\begin{split}
 u(t)&=e^{t\Delta}[u_0+t\P\theta_0e_3]-B_1(u,u)-B_2(u,\theta)\\
 &=\biggl(\int\theta_0\Bigr)t\,K(t)+t\,F(t)*V+e^{t\Delta}u_0-[B_1(u,u)+B_2(u,\theta)],
\end{split}
\end{equation}
where the notations here are the same as before.
Fix $A>0$ large enough.
Multiplying by $t^{\frac12(1-3/p)}$, taking the $L^p(B_{A\sqrt t}^c)$-norm and applying the triangle inequality,
leads to
\[
\begin{split}
t^{\frac12(1-3/p)}&\|u(t)\|_{L^p(B_{A\sqrt t}^c)}\\
&\ge   t^{\frac12(1-3/p)}\biggl[
        \,\bigl|\textstyle\int\theta_0\bigr|\,t\,\|K(t)\|_{L^p(B_{A\sqrt t}^c)}
        -\|B_1(u,u)(t)+B_2(u,\theta)(t)\|_{L^p(B_{A\sqrt t}^c)}+o(1)\biggr]\\
&\ge c\bigl|\textstyle\int\theta_0\bigr|A^{-3+3/p}-CA^{-4+3/p}+o(1)\\
&\ge \textstyle \frac{c}{2}\bigl|\textstyle\int\theta_0\bigr| A^{-3+3/p} +o(1). 
\end{split}
\]
Here the $o(1)$-term (as $t\to+\infty$) includes the contributions of 
two terms $F(t)*V$ and $e^{t\Delta}u_0$, that can indeed be neglected accordingly to
Lemma~\ref{lem:fv-r} and Lemma~\ref{lem:hu}.
Here we applied also  Lemma~\ref{lem:b12uu} and  Lemma~\ref{lem:mainte}.
From the above inequalities we deduce that, for $t>0$ large enough,
\[
t^{\frac12(1-3/p)}\|u(t)\|_{L^p}\ge \frac{c}{4}\bigl|\textstyle\int\theta_0\bigr| A^{-3+3/p}.
\]
The assertion of the theorem is now proved.
\end{proof}

\begin{remark}
The result of the present paper seems to be quite specific of the fact that we set the problem in the whole space. In bounded domains exponential decays are expected due to Poincar\'e's inequality. More interesting is the case of other unbounded domains: in the half-space case with Dirichlet boundary conditions the energy does not grow large anymore, but it decays to zero. See~\cite{HanS14}. 
So the long time behavior of Boussinesq flows 
in the half-space and in the whole space are completely different.
From the physical point of view, this different behavior seems to be related to the energy dissipation that
occurs in the boundary layer. Indeed, experimental analyses of energy spectra of viscous flows are available and put in evidence that walls selectively damp out the higher frequencies: the low-frequency content increases and the high-frequency content decreases as the boundary layer is traversed from the freestream to the wall. See, e.g., \cite{HerW07}.
In the case of the exterior domain problem, the long time behavior of the energy is addressed in~\cite{Han12}: therein, the decay of the energy is obtained under additional (non-generic) conditions on the temperature; in the absence of such conditions energy growth is expected, because there is not enough boundary to maintain the dissipation mechanisms, but this has not been rigorously established yet.
\end{remark}

\section{Acknowledgements}
The authors would like to thank the referees for their careful reading and useful suggestions, that have been incorporated in this revised version of the manuscript.

\end{document}